\documentclass[11pt,leqno,a4paper]{amsart}
\usepackage{amsmath}
\usepackage{amsbsy}
\usepackage{amsthm}
\usepackage{amssymb}
\usepackage{amscd}
\usepackage{mathabx}
\usepackage{accents}

\usepackage{xcolor}

\usepackage{tikz}
\usetikzlibrary{matrix,arrows,decorations.pathmorphing}
\usepackage{tikz-cd}

\usepackage{enumerate}
\usepackage{verbatim}

\usepackage{hyperref}
\usepackage{mathabx}

\usepackage[english]{babel}
\usepackage[utf8]{inputenc}
\usepackage[totalheight=22 true cm, totalwidth=16 true cm]{geometry}
\usepackage{mathrsfs}
\usepackage[all]{xy}

\newtheorem{thm}{Theorem}[section]
\newtheorem{cor}[thm]{Corollary}
\newtheorem{lemma}[thm]{Lemma}

\newtheorem{prop}[thm]{Proposition}

\theoremstyle{remark}

\theoremstyle{definition}

\newtheorem{defn}[thm]{Definition}

\newtheorem{rmk}[thm]{Remark}
\newtheorem{exa}[thm]{Example}

\newtheorem{notation}[thm]{Notation}
\numberwithin{equation}{thm}

\def\beq{\begin{equation}}
\def\eeq{\end{equation}}

\def\beqn{\begin{equation*}}
\def\eeqn{\end{equation*}}

\def\ben{\begin{enumerate}}
\def\een{\end{enumerate}}

\def\crash#1{}

\def\L{{\mathbb L}}

\def\cB{{\mathcal B}}
\def\cC{{\mathcal C}}
\def\cD{{\mathcal D}}
\def\cE{{\mathcal E}}
\def\cF{{\mathcal F}}

\def\cL{{\mathcal L}}

\def\cR{{\mathcal R}}
\def\cS{{\mathcal S}}

\def\cX{{\mathcal X}}

\def\bC{{\mathbf C}}

\def\Spec{{\rm Spec \,}}
\def\Hom{{\rm Hom \,}}

\def\id{{\rm id\,}}
\def\ker{{\rm Ker\,}}

\def\id{{\rm Id\,}}

\def\Map{{\rm Map\,}}

\def\solid{{\rm solid\,}}
\def\Ho{{\rm Ho\,}}

\def\CondAb{\mathbf{CondAb}}
\def\bCat{\mathbf{Cat}}

\def\bMod{\mathbf{Mod}}

\def\bCBorn{\mathbf{CBorn}}

\def\bAlg{\mathbf{Alg}}
\def\bFree{\mathbf{Free}}
\def\bFin{\mathbf{Fin}}

\author{Federico Bambozzi}

\title{Homological and homotopical epimorphisms in monoidal $\infty$-categories}

\thanks{The author was supported by the following grants: UNIPD DOR2315777/23 "Algebraic Geometry and Arithmetic Geometry"; MUR PRIN
2022 "The Arithmetic of Motives and L-functions". Part of this work was carried out during a Research-in-Groups programme funded by the International Centre for Mathematical Sciences, Edinburgh.}

\begin{document}

\begin{abstract}
We compare homological and homotopical epimorphisms of associative algebras in monoidal $\infty$-categories. Assume that geometric realizations and pushouts of associative algebras exist and that tensoring preserves geometric realizations. Then every homological epimorphism is homotopical. If, in addition, the ambient category is additive and tensoring preserves finite coproducts, the converse holds. This extends the result of Chuang and Lazarev for dg-algebras. In the symmetric monoidal setting, we discuss consequences for centers and commutative refinements, as well as applications to derived analytic geometry.
\end{abstract}

\maketitle

\tableofcontents

\section{Introduction}

Let $f:A\to B$ be a morphism of associative algebras in a monoidal $\infty$-category $\cC$. The canonical maps
\[
B\ast_A B\longrightarrow B, \qquad B\otimes_A B\longrightarrow B
\]
give rise to two notions of epimorphism. We call $f$ \emph{homotopical} if the first map is an equivalence and \emph{homological} if the second is an equivalence. Here $\ast_A$ denotes the pushout of associative algebras and $\otimes_A$ the relative tensor product of modules. Chuang and Lazarev prove that these notions agree for unbounded dg-algebras \cite[Theorem~4.4]{ChuangLazarev} and ask whether this equivalence reflects a general categorical phenomenon \cite[Remark~4.2]{ChuangLazarev}. We show that the equivalence can fail for associative algebras in spaces, equipped with the cartesian monoidal structure (see Remark~\ref{rmk:additivity_necessary}). Theorem~\ref{thm:main} establishes the equivalence in additive monoidal $\infty$-categories, provided the relevant algebra pushouts and geometric realizations exist and tensoring preserves finite coproducts and geometric realizations.

Our motivation comes from bornological algebra and analytic geometry. See \cite{BambozziBenBassatDagger,BenBassatKremnizerRAG, BambozziBenBassatKremnizerStein, BambozziBenBassatKremnizerFOne,BambozziKremnizerSheafyness} for the bornological approach to derived analytic geometry, and \cite{BenBassatMukherjee,KKM-HKR,BenBassatKellyKremnizer, KellyMukherjeeLocalising, BambozziKremnizerSheafyness, BambozziMiharaZValued, BambozziChiarellottoVanniTempered} for further developments. In derived analytic geometry, localizations are naturally expressed by a tensor idempotence condition
\[ B\widehat\otimes_A^{\mathbb L} B\simeq B. \]
This condition occurs, for instance, in the homological description of affinoid, dagger, and Stein localizations \cite{BenBassatKremnizerRAG,BambozziBenBassatDagger, BambozziBenBassatKremnizerStein,BambozziKremnizerSheafyness, BenBassatMukherjee}. See Section~\ref{sec:applications} for a more detailed account. On the other hand, in the noncommutative geometric framework of \cite{BambozziCapoferriMurro}, localization is formulated intrinsically in the category of associative algebras through the condition
\[
B\ast_A B\simeq B.
\]
The equivalence between the pushout and tensor conditions is a key input in that work. Extending this localization formalism from the algebraic to the analytic  setting therefore requires understanding when the two conditions agree. Establishing this comparison in a general monoidal setting is the main goal of this paper.

The paper is organized as follows. Section~2 collects the preliminaries on algebras, modules, central structures, square-zero extensions, and derivations. Section~3 proves our main theorem. The proof that homotopical epimorphisms are homological epimorphisms is inspired by Silver's argument for ordinary rings \cite{SIL}: we construct a relative derivation with values in $B\otimes_A B$ and show that its triviality forces relative multiplication to be an equivalence. Section~4 discusses the cotangent complex, examples motivating the theorem, and geometric applications. Appendix~\ref{app:closed_proof} gives an alternative proof under the additional assumption that $\cC$ is closed.

\subsection*{Acknowledgements}
\addtocontents{toc}{\protect\setcounter{tocdepth}{1}}
We would like to thank Lidia Angeleri, Jack Kelly, Kobi Kremnizer, Devarshi Mukherjee, Simone Murro, Filippo Papallo, Stefano Rosarin, Giuseppe Rosolini for many helpful discussions on topics related to this paper.

\subsection*{Notation}
\addtocontents{toc}{\protect\setcounter{tocdepth}{1}}

We follow the conventions for simplicial sets and $\infty$-categories established in \cite{HTT}. Let $\Delta$ denote the simplex category, and let
\[
\mathbf{sSet}=\operatorname{Fun}(\Delta^{\mathrm{op}},\mathbf{Set})
\]
denote the category of simplicial sets. For each $n\geq 0$, we write $\Delta^n$ for the standard $n$-simplex and, for $0\leq i\leq n$, we denote by
\[
\Lambda_i^n\hookrightarrow\Delta^n
\]
the inclusion of the $i$-th horn. A simplicial set $\cX_\bullet$ is an \emph{$\infty$-category}, or equivalently a \emph{weak Kan complex}, if every inner horn in $\cX_\bullet$ admits a filler, that is, if it has the right lifting property with respect to $\Lambda_i^n\hookrightarrow\Delta^n$ whenever $0<i<n$; see \cite[Definition~1.1.2.4]{HTT}.

Given two objects $x,y\in\cX_0$, we denote by $\Map_{\cX}(x,y)$ the corresponding mapping space, in the sense of \cite[Definition~1.2.2.1]{HTT}. For an $\infty$-category $\cC$, we write $X\in\cC$ to mean that $X$ is an object of $\cC$. We denote by $\bCat$ the $\infty$-category of small $\infty$-categories. Finally, if two objects $x,y\in\cC$ are equivalent, we write $x\simeq y$.

An $\infty$-category is called \emph{additive} if it has finite biproducts and its canonical commutative monoid structures on mapping spaces are group-like. In particular, subtraction of morphisms is defined coherently. Additivity of $\cC$ does not by itself assert that a chosen tensor product preserves these biproducts; this compatibility will always be stated separately. All fibers, pullbacks, and mapping spaces below are understood in the $\infty$-categorical sense.

Throughout this paper, we use cohomological conventions for complexes.

\section{\texorpdfstring{Monoidal $\infty$-categories}{Monoidal infinity-categories}}

In this section, we recall the basic definitions concerning monoidal
$\infty$-categories and algebra objects and fix the notation used
throughout the paper. We assume familiarity with
$\infty$-categories and coCartesian fibrations, as developed in
\cite{HTT,LandM}, and follow the operadic conventions of \cite{HA}.

Let $\bFin_*$ denote the nerve of the category of finite pointed
sets, and let $\mathrm{Assoc}^\otimes$ denote the associative
$\infty$-operad with its canonical map to $\bFin_*$.

\begin{defn}\label{defn:symmetric_monoidal}
A \emph{monoidal $\infty$-category} is a coCartesian fibration of $\infty$-operads
\[
p:\cC^\otimes\longrightarrow\mathrm{Assoc}^\otimes.
\]
Its underlying $\infty$-category $\cC$ is the fiber over the unique color of $\mathrm{Assoc}^\otimes$. The operadic Segal condition identifies the fiber over $\langle n\rangle$ with $\cC^n$. CoCartesian transport along the ordered binary multiplication and the nullary operation determines the tensor product $\otimes:\cC\times\cC\to\cC$ and its unit $1_\cC$ (see \cite[\S4.1.1]{HA}).

A \emph{symmetric monoidal $\infty$-category} is a coCartesian fibration of $\infty$-operads $\cC^\otimes_{\mathrm{sym}}\to\bFin_*$ (see \cite[Definition~2.0.0.7]{HA}). It determines a monoidal $\infty$-category by the pullback
\[
\cC^\otimes
=\mathrm{Assoc}^\otimes\times_{\bFin_*}
\cC^\otimes_{\mathrm{sym}}.
\]
\end{defn}

\begin{defn}\label{defn:category_algebras}
For a monoidal $\infty$-category
$p:\cC^\otimes\to\mathrm{Assoc}^\otimes$, define
\[
\bAlg(\cC)
=\operatorname{Fun}^{\mathrm{inert}}_{/\mathrm{Assoc}^\otimes}
  (\mathrm{Assoc}^\otimes,\cC^\otimes),
\]
the full subcategory of the functor $\infty$-category over
$\mathrm{Assoc}^\otimes$ spanned by functors preserving inert
morphisms. Thus its objects are inert-preserving sections of $p$,
called \emph{associative algebra objects} in $\cC$.
\end{defn}

Although stated in the language of $\infty$-operads, Definition \ref{defn:category_algebras} is a direct homotopical generalization of the notion of an associative algebra over a monoidal category. We emphasize this further with the following two key examples.

\begin{exa}
For a commutative dg-algebra $A$ over a commutative ring $k$, let $\bMod_A$ denote the derived $\infty$-category of dg $A$-modules, obtained by inverting quasi-isomorphisms and equipped with the derived tensor product. Then $\bAlg(\bMod_A)$ is modeled by associative dg $A$-algebras with quasi-isomorphisms inverted. If $\mathbf{Sp}$ denotes spectra with the smash product, $\bAlg(\mathbf{Sp})$ is the $\infty$-category of associative ring spectra.
\end{exa}

The following is a standard result.

\begin{lemma} \label{lemma:free_algebras}
Let $(\cC, \otimes)$ be a monoidal $\infty$-category with countable coproducts, and assume that $\otimes$ preserves countable coproducts in each variable. Then the forgetful functor $\bAlg(\cC) \to \cC$ admits a left adjoint, called the \emph{free algebra functor}.
\end{lemma}
\begin{proof}
For $X \in \cC$, the free algebra has underlying object
\[
\bFree_\cC(X)=\coprod_{n\geq0}X^{\otimes n},
\qquad X^{\otimes0}=1_\cC.
\]
Concatenation defines its multiplication and the $n=0$ summand defines its unit, with the coherences inherited from the monoidal structure. See \cite[Theorem~1.4.2]{DAGII}.
\end{proof}

We will also use the following lemma.

\begin{lemma}\label{lemma:lax_ff_alg}
Let
\[
F^\otimes:\cC^\otimes\longrightarrow\cD^\otimes
\]
be a lax monoidal functor between monoidal $\infty$-categories, and let
\[
F:\cC\longrightarrow\cD
\]
be the underlying functor. Assume that
\begin{itemize}
\item[-] $F$ is fully faithful;
\item[-] for every $X,Y\in\cC$, the lax monoidal structure map
\[
F(X)\otimes F(Y)\longrightarrow F(X\otimes Y)
\]
is an equivalence.
\end{itemize}
Then the induced functor
\[
\bAlg(\cC)\longrightarrow\bAlg(\cD)
\]
is fully faithful.
\end{lemma}

\begin{rmk}
We do not assume that the unit map
\[
1_{\cD}\longrightarrow F(1_{\cC})
\]
is an equivalence. Thus $F$ need not be strongly monoidal: under the
hypotheses of the lemma, the only possible failure of strong monoidality
occurs at the unit.
\end{rmk}

\begin{proof}
We first show that the induced functor on nonunital associative
algebras is fully faithful. Let $\cE\subseteq\cD$ be the full
essential image of $F$. Since the tensor comparison maps are
equivalences, $\cE$ is closed under tensor products. Applying
\cite[Proposition~2.2.1.1(1)--(2)]{HA} over the nonunital
associative operad $\mathrm{Assoc}^{\otimes}_{\mathrm{nu}}$,
we equip $\cE$ with the inherited nonunital monoidal structure.
The functor $F:\cC\to\cE$ is then a nonunital monoidal functor
whose underlying functor is an equivalence. Hence it is a
nonunital monoidal equivalence by
\cite[Remark~2.1.3.8]{HA}, and therefore induces an equivalence
\[
\bAlg^{\mathrm{nu}}(\cC)
\simeq
\bAlg^{\mathrm{nu}}(\cE).
\]
Moreover, the full inclusion $\cE\subseteq\cD$ induces a fully
faithful functor on nonunital associative algebras, directly
from their definition as operadic functors (see
\cite[Definition~2.1.2.7]{HA} for the definition of morphism between $\infty$-operads). Composing these functors proves
that
\[
\bAlg^{\mathrm{nu}}(\cC)
\longrightarrow
\bAlg^{\mathrm{nu}}(\cD)
\]
is fully faithful.

By Lurie's quasi-unit theorem
\cite[Theorem~5.4.3.5 and Corollary~5.4.3.6]{HA},
the mapping space of unital algebra morphisms between $A$ and $B$
identifies with the union of components of the nonunital algebra
mapping space represented by maps $g$ satisfying
$g\eta_A\simeq\eta_B$.
Since full faithfulness on nonunital algebras has already been
proved, it remains only to show that $F$ detects this condition.

Put $E=F(1_\cC)$ and let $u:1_\cD\to E$ be its unit.
The lax monoidal structure equips each $F(X)$ with a canonical
left $E$-module structure, functorially in $X$, through
\[
E\otimes F(X)
\longrightarrow F(1_\cC\otimes X)
\simeq F(X).
\]
In particular, applying $F$ to any morphism in $\cC$ gives an
$E$-linear morphism. Let $A,B$ be unital associative algebras in $\cC$, let
$g:A\to B$ be a morphism of their underlying nonunital algebras, and $\eta_A: 1_\cC \to A$ and $\eta_B: 1_\cC \to B$ their units.
The maps
\[
\alpha = F(g\eta_A),
\qquad
\beta = F(\eta_B)
\]
are therefore $E$-linear morphisms from $E$ to $F(B)$.
Since the regular left $E$-module is free on $1_\cD$,
precomposition with $u$ induces an equivalence
\[
u^*:
\Map_{{}_E\bMod}(E,F(B))
\xrightarrow{\;\simeq\;}
\Map_\cD(1_\cD,F(B));
\]
see \cite[Corollary~4.2.4.7]{HA}. 
The units of the induced algebras are
\[
\eta_{F(A)}=F(\eta_A)\circ u,
\qquad
\eta_{F(B)}=F(\eta_B)\circ u.
\]
Hence, if $F(g)$ preserves the unit up to homotopy, then
\[
u^*(\alpha)
=
F(g)\circ\eta_{F(A)}
\simeq
\eta_{F(B)}
=
u^*(\beta).
\]
The equivalence $u^*$ gives $\alpha\simeq\beta$ as
$E$-linear morphisms. Forgetting the module structure and
using full faithfulness of $F$, we obtain
\[
g\eta_A\simeq\eta_B.
\]
Conversely, this homotopy immediately implies that $F(g)$
preserves the unit up to homotopy. Thus $F$ preserves and
detects this property.

To pass from this condition to coherent unital algebra
morphisms, we use Lurie's quasi-unit theorem
\cite[Theorem~5.4.3.5 and Corollary~5.4.3.6]{HA}.
A quasi-unit of a nonunital associative algebra is a morphism
from the monoidal unit whose left and right multiplication
maps are homotopic to the identity. For an algebra already
equipped with a unit, every quasi-unit is homotopic to that
unit. Consequently, a nonunital algebra morphism $g:A\to B$
is quasi-unital precisely when $g\eta_A\simeq\eta_B$.

Thus the equivalence of nonunital mapping spaces restricts to
an equivalence of unital mapping spaces, proving the claim.
\end{proof}

We now recall the module operads associated with an algebra object.

\begin{notation}\label{not:left_module_operad}
Let $\mathcal L\mathrm{Mod}^\otimes$ and $\mathcal R\mathrm{Mod}^\otimes$ denote the left- and right-module $\infty$-operads of \cite[\S4.2.1]{HA}. Their canonical input orders determine maps to $\mathrm{Assoc}^\otimes$: the module input is last in a left-module operation and first in a right-module operation. We write
\[
\cL\mathrm{Mod}(\cC)
=\operatorname{Fun}^{\mathrm{inert}}_{/\mathrm{Assoc}^\otimes}
 (\mathcal L\mathrm{Mod}^\otimes,\cC^\otimes)
\]
and define $\cR\mathrm{Mod}(\cC)$ analogously using $\mathcal R\mathrm{Mod}^\otimes$. Here the superscript $\mathrm{inert}$ denotes the full subcategory spanned by functors preserving inert morphisms (see \cite[Example~4.2.1.16]{HA}). Restriction along the algebra-color inclusion $\mathrm{Assoc}^\otimes\hookrightarrow \mathcal L\mathrm{Mod}^\otimes$ gives $\cL\mathrm{Mod}(\cC)\to\bAlg(\cC)$, and similarly for right modules.
\end{notation}

\begin{defn} \label{defn:left_right_modules}
    Let $A\in\bAlg(\cC)$. The \emph{$\infty$-category of left $A$-modules}, denoted ${}_A\bMod$, is the fiber over $A$ of the canonical functor  $\cL\mathrm{Mod}(\cC)\longrightarrow\bAlg(\cC)$. Similarly, the \emph{$\infty$-category of right $A$-modules}, denoted $\bMod_A$, is the fiber over $A$ of $\cR\mathrm{Mod}(\cC)\longrightarrow\bAlg(\cC)$.
\end{defn}

We now recall the notion of bimodule.

\begin{notation}\label{not:bimodule_operad}
Let $\cB\mathrm{Mod}^\otimes$ denote the bimodule
$\infty$-operad of \cite[\S4.3.1]{HA}.
Its canonical map to $\mathrm{Assoc}^\otimes$ records the
input order: the left algebra inputs precede the module input,
and the right algebra inputs follow it. We set
\[
\cB\mathrm{Mod}(\cC)
=\operatorname{Fun}^{\mathrm{inert}}_{/\mathrm{Assoc}^\otimes}
 (\cB\mathrm{Mod}^\otimes,\cC^\otimes);
\]
see \cite[Example~4.3.1.15]{HA}.
\end{notation}

Algebras over $\cB\mathrm{Mod}^{\otimes}$ encode triples $(A_-,M,A_+)$ consisting of two associative algebras and an $A_-$-$A_+$-bimodule, see \cite[Remark~4.3.1.14]{HA}.

\begin{rmk}
The bimodule operad contains the left- and right-module operads and two copies of the associative operad, selecting the two algebra objects \cite[Remarks~4.3.1.9--4.3.1.10]{HA}. Bimodules can also be described as right-module objects in the category of left modules, using its right $\cC$-action \cite[Theorem~4.3.2.7]{HA}. 
\end{rmk}

For $A,B\in\bAlg(\cC)$, we denote by ${}_A\bMod_B$ the fiber of $ \cB\mathrm{Mod}(\cC)\longrightarrow \bAlg(\cC)\times\bAlg(\cC)$ over $(A,B)$. Thus ${}_A\bMod_B$ is the $\infty$-category of $A$-$B$-bimodules in $\cC$. When $A=B$, assume that $\cC$ admits geometric realizations of simplicial objects and that $\otimes$ preserves them separately in each variable. Then ${}_A\bMod_A$ carries a monoidal structure whose tensor product is the relative tensor product $-\otimes_A-$, see \cite[Proposition~4.4.3.12]{HA}.

\begin{prop} \label{prop:bimodules_forget_conservative}
    The forgetful functors ${}_A \bMod_A \to \bMod_A$ and ${}_A \bMod_A \to {}_A \bMod$ are conservative.
\end{prop}
\begin{proof}
By \cite[Corollary 4.3.3.2 and Example 4.7.3.9]{HA}, the forgetful functor
\[
{}_A\bMod_A \to \cC
\]
is conservative. This functor factors through each of the two forgetful
functors in the statement. Hence, if a morphism of $A$-bimodules becomes
an equivalence after forgetting either the left or the right action, then
its underlying morphism in $\cC$ is an equivalence, and therefore the
original morphism of bimodules is an equivalence.
\end{proof}

For $A\in\bAlg(\cC)$, we denote by
\[
\bAlg(\cC)_{A/}
\]
the undercategory of $A$ in $\bAlg(\cC)$, whose objects will also be
called \emph{$A$-algebras}.

\begin{prop}\label{prop:algebras_bimodules}
Let $(\cC,\otimes)$ be a monoidal $\infty$-category, and let
$A \in \bAlg(\cC)$. Assume that $\cC$ admits geometric realizations of
simplicial objects and that $\otimes$ preserves geometric realizations
separately in each variable. Then there is an equivalence of
$\infty$-categories
\[
\bAlg({}_A\bMod_A)\simeq \bAlg(\cC)_{A/},
\]
where ${}_A\bMod_A$ is endowed with its monoidal structure given by the
relative tensor product $\otimes_A$.
\end{prop}

\begin{proof}
Theorem~4.4.1.28 of \cite{HA} identifies $A$-bimodules with associative operadic modules over $A$. Under the stated hypotheses on geometric realizations, their relative tensor product gives the monoidal structure of \cite[Proposition~4.4.3.12]{HA}. The associative operad is coherent \cite[Proposition~4.1.1.20]{HA}, so \cite[Corollary~3.4.1.7]{HA} gives the asserted equivalence. Under it an algebra in $A$-bimodules has unit $A\to B$ and multiplication $B\otimes_A B\to B$, and the associated associative algebra in $\cC$ is equipped with that unit map from $A$.
\end{proof}

We will also need the notion of center of an algebra. Notice that if $\cC$ is closed symmetric monoidal and admits
totalizations of cosimplicial objects, then the
$\infty$-category ${}_B\bMod_B$ is naturally enriched
over $\cC$. Indeed, its internal mapping objects can be
constructed by totalizing the cosimplicial internal Hom
objects associated with the bar resolution of a
bimodule (see \cite[\S\S4.3.3 and 4.4.2]{HA}).

\begin{defn} \label{defn:center}
Let \(B\in \bAlg(\cC)\). We define the \emph{center space} of \(B\) as
\[
\mathcal Z(B) = \Map_{{}_B\bMod_B}(B,B).
\]
If $\cC$ is closed symmetric monoidal and admits
totalizations of cosimplicial objects, the \emph{center object} (or just \emph{center}) of \(B\) is defined by
\[
\underline{\mathcal Z}(B)=\underline{\Hom}_{{}_B\bMod_B}(B, B),
\]
the internal endomorphism object of \(B\) viewed as a \(B\)-\(B\)-bimodule.
\end{defn}

\begin{rmk}\label{rmk:center_E_2}
Under the hypotheses ensuring the enriched mapping objects above, and assuming that the relative tensor product commutes with geometric realizations, the center object carries a canonical $E_2$-algebra structure: it is the endomorphism object of the monoidal unit $B$ in $({}_B\bMod_B,\otimes_B)$. See \cite[Theorem~5.3.1.30]{HA} and Dunn additivity \cite[Theorem~5.1.2.2]{HA}. More generally, if $\cC$ is presentable symmetric monoidal and its tensor product preserves small colimits separately, the center of an $E_k$-algebra exists and carries a canonical $E_{k+1}$-algebra structure, see
\cite[Corollary~5.3.1.15]{HA}.
\end{rmk}

For ordinary associative algebras over a commutative ring $A$, centrality of a morphism $f:A\to B$ is the condition that its image lie in the center of $B$. Central $A$-algebras therefore form a full subcategory of the category of associative algebras under $A$. In the homotopical setting, a central structure includes the coherent compatibility of the $A$-action with the algebra structure on $B$. We must therefore distinguish an associative algebra under $A$ from one equipped with a chosen central structure. The comparison functor between central algebras and algebras forgets this structure and therefore need not be fully faithful: morphisms of central $A$-algebras are required to preserve it coherently. We now impose the further hypothesis that $\cC$ is symmetric monoidal.

\begin{defn}\label{defn:central_algebra}
Let $A\in\bC\bAlg(\cC)$, and suppose that the relative tensor products defining the symmetric monoidal category $(\bMod_A,\otimes_A)$ exist. A \emph{central $A$-algebra} is an associative algebra object of $(\bMod_A,\otimes_A)$. Such an object has an underlying associative algebra $B$ in $\cC$, and its unit induces an associative algebra morphism $f:A\to B$. This construction defines a natural functor
\[
\Phi_A:
\bAlg(\bMod_A)\longrightarrow\bAlg(\cC)_{A/}.
\]
For a specified associative algebra morphism $f:A\to B$,
a \emph{central structure on $f$} is an object
$\widetilde B\in\bAlg(\bMod_A)$ together with an equivalence
\[
\Phi_A(\widetilde B)\simeq(f:A\to B)
\]
in $\bAlg(\cC)_{A/}$.
\end{defn}

In particular, centrality supplies a coherent identification of the two $A$-actions
\[
A\otimes B \xrightarrow{f\otimes\id}B\otimes B\xrightarrow{\mu_B}B,
\qquad
A\otimes B\xrightarrow{s_{A,B}}B\otimes A
 \xrightarrow{\id\otimes f}B\otimes B\xrightarrow{\mu_B}B.
\]

We give an example of two different central structures on a dg-algebra over a field.

\begin{exa}
Let $k$ be a field of characteristic zero, and put
\[
A=k[t],
\qquad
B=k[x]\otimes_k\Lambda_k(\epsilon),
\qquad
|t|=|x|=0,\quad |\epsilon|=-1,
\]
with zero differentials. Fix the associative algebra morphism
\[
f:A\longrightarrow B,\qquad t\longmapsto0.
\]

For $\lambda\in\{0,1\}$, define a dg $A$-algebra $C_\lambda$
by adjoining generators $x,\epsilon,h$, with
\[
|x|=0,\qquad |\epsilon|=|h|=-1,
\qquad
d(x)=d(\epsilon)=0,\quad d(h)=t.
\]
The element $t$ is central, and the remaining relations are
\[
x\epsilon=\epsilon x,\qquad
\epsilon^2=h^2=0,\qquad
h\epsilon=-\epsilon h,\qquad
hx-xh=\lambda\epsilon.
\]
The differential respects these relations. Since $t$ is
central, each $C_\lambda$ is an associative algebra object
of $\bMod_A$. The relations give 
\[
C_\lambda
\cong
A[x]\oplus A[x]\epsilon
\oplus A[x]h\oplus A[x]\epsilon h
\]
as graded $A$-modules. Thus, as complexes,
\[
C_\lambda
\cong
B\otimes_k
\bigl(k[t]\otimes_k\Lambda_k(h),\ d(h)=t\bigr).
\]
It follows that the inclusion
\[
i_\lambda:B\longrightarrow C_\lambda
\]
is a quasi-isomorphism of dg $k$-algebras.
Moreover, the structure map $A\to C_\lambda$ is homotopic,
as an associative algebra morphism, to the composite
$A\xrightarrow{f}B\xrightarrow{i_\lambda}C_\lambda$.
Indeed, $A=k[t]=k\langle t\rangle$ is free as an associative
algebra on one degree-zero generator, and $d(h)=t$
supplies the required nullhomotopy of its image.
Thus, after choosing this homotopy, we obtain an equivalence
\[
\Phi_A(C_\lambda)\simeq(f:A\to B)
\]
in $\bAlg(D(k))_{A/}$.
Both $C_0$ and $C_1$ therefore define central structures
on the same morphism $f$.
These central structures are inequivalent. To see this,
apply derived extension of scalars along $A\to k$,
$t\mapsto0$. Each $C_\lambda$ is a bounded complex of
free $A$-modules, so
\[
k\otimes_A^{\mathbb L}C_\lambda
\simeq
C_\lambda/(t).
\]
These quotient dg-algebras have zero differential.
For $\lambda=0$, the quotient is graded-commutative.
For $\lambda=1$, it satisfies
\[
hx-xh=\epsilon\neq0,
\]
and is therefore not graded-commutative.
Their graded homology algebras are therefore
nonisomorphic and $C_0$ and $C_1$ cannot be equivalent in $\bAlg(\bMod_A)$.
\end{exa}

\begin{lemma}\label{lemma:center_evaluation}
Let $B\in\bAlg(\cC)$, and assume that the enriched bimodule
endomorphism object
\[
\underline{\mathcal Z}(B)
=
\underline{\Hom}_{{}_B\bMod_B}(B,B)
\]
exists. Regard it as an associative algebra by composition.
There is a canonical morphism of associative algebras
\[
e_B:\underline{\mathcal Z}(B)\longrightarrow B,
\]
whose underlying morphism is
\[
\underline{\mathcal Z}(B)
\simeq
\underline{\mathcal Z}(B)\otimes1_\cC
\xrightarrow{\id\otimes\eta_B}
\underline{\mathcal Z}(B)\otimes B
\xrightarrow{\operatorname{ev}}
B.
\]
We call $e_B$ evaluation at the unit of $B$.
\end{lemma}

\begin{proof}
We first recall the meaning of the evaluation morphism.
The enriched endomorphism object is characterized by natural
equivalences
\[
\Map_\cC\bigl(X,\underline{\mathcal Z}(B)\bigr)
\simeq
\Map_{{}_B\bMod_B}(X\otimes B,B),
\]
where $X\otimes B$ carries the bimodule structure induced
by the regular bimodule $B$.
Taking $X=\underline{\mathcal Z}(B)$, the identity morphism
of $\underline{\mathcal Z}(B)$ corresponds to a bimodule
morphism
\[
\operatorname{ev}:
\underline{\mathcal Z}(B)\otimes B\longrightarrow B.
\]
To prove that $e_B$ is an algebra morphism, we notice that we have a canonical morphism
\[
r:
\underline{\mathcal Z}(B)
\longrightarrow
R_B:=
\underline{\Hom}_{\bMod_B}(B,B).
\]
where the object $R_B$ is canonically equivalent to $B$ by left multiplication. It is straightforward to check that this is a morphism of algebras because composition of morphisms of bimodules and right modules are compatible.
\end{proof}

\begin{prop}\label{prop:central_algebra_definition}
Assume that $\cC$ is presentable symmetric monoidal, satisfies the hypothesis of Lemma \ref{lemma:center_evaluation} and that the tensor product commutes with small colimits in each variable. For $A\in\bC\bAlg(\cC)$ and $f:A\to B$, the space of central structures on $f$ is equivalent to the homotopy fiber at $f$ of
\[
\Map_{\operatorname{Alg}_{E_2}(\cC)}
       (A,\underline{\mathcal Z}(B))
\longrightarrow
\Map_{\bAlg(\cC)}(A,B),
\qquad g\longmapsto e_Bg.
\]
Here $A$ is regarded as an $E_2$-algebra by restriction of its commutative structure.
\end{prop}
\begin{proof}
A central action consists of an $E_2$-algebra acting on an associative algebra by a compatible algebra structure on its module object. Equivalently, it is an algebra over the operadic tensor product of the associative and left-module operads. The universal property of the $E_2$-center classifies these coherent actions by $E_2$-maps into $\underline{\mathcal Z}(B)$; see \cite[\S5.3.1]{HA} and \cite[\S2.1]{LurieRotation}. For the fixed commutative algebra $A$, these are precisely associative algebra structures in $\bMod_A$. Evaluating an action on the unit of $B$ gives its underlying morphism $A\to B$. Taking the homotopy fiber of the classification over $f$ proves the statement.
\end{proof}

\begin{prop}\label{prop:Alg_ModA_central}
The functor $\Phi_A$ identifies $\bAlg(\bMod_A)$ with the $\infty$-category of $A$-algebras equipped with central structures. The multiplication of such an algebra is
\[
B\otimes_A B\longrightarrow B
\]
in $\bMod_A$, and its underlying multiplication in $\cC$ is obtained by precomposing with $B\otimes B\to B\otimes_A B$.
\end{prop}
\begin{proof}
This is the operadic description of algebra objects in the symmetric monoidal category of $A$-modules. The unit is an $A$-linear map $A\to B$, and every operation is $A$-linear in the sense encoded by $\otimes_A$. Forgetting this structure gives $\Phi_A$. 
\end{proof}

We now return to assume that $\cC$ is a general monoidal $\infty$-category.
We write $B\ast_A C$ for the pushout of a span $B\leftarrow A\to C$ in $\bAlg(\cC)$ whenever it exists. For instance, the pushouts used below exist if $\cC$ is presentable and tensoring preserves small colimits separately \cite[Corollary~3.2.3.3]{HA}.

We define derivations as sections of split square-zero extensions. Throughout this discussion, $\cC$ is additive and its tensor product preserves finite coproducts separately.
For a $B$-bimodule $M$, the split square-zero algebra $B\oplus M$ has unit $(\eta_B,0)$. Its multiplication, on the ordered summands
\[
(B\oplus M)\otimes(B\oplus M)
\simeq (B\otimes B)\oplus(B\otimes M)
       \oplus(M\otimes B)\oplus(M\otimes M),
\]
is given respectively by the multiplication of $B$, the left $B$-action, the right $B$-action, and zero, with the first map landing in $B$ and the other three in $M$.
The algebra and bimodule coherences give the associative algebra structure. 

\begin{defn}\label{defn:derivations}
Let $B\in\bAlg(\cC)$ and $M\in{}_B\bMod_B$, and write
\[
p_M:B\oplus M\longrightarrow B
\]
for the augmentation of the split square-zero extension.
A \emph{derivation of $B$ into $M$} is a section of $p_M$
in $\bAlg(\cC)$. Its space is
\[
\operatorname{Der}(B,M)
=
\operatorname{fib}_{\id_B}
\left(
\Map_{\bAlg(\cC)}(B,B\oplus M)
\xrightarrow{p_M\circ-}
\Map_{\bAlg(\cC)}(B,B)
\right).
\]
Thus a derivation consists of an algebra morphism
$s:B\to B\oplus M$ together with an identification
$p_Ms\simeq\id_B$.

For an algebra morphism $f:A\to B$, an
\emph{$A$-relative derivation} is a section of $p_M$
in $\bAlg(\cC)_{A/}$. Its space is
\[
\operatorname{Der}_A(B,M)
=
\operatorname{fib}_{\id_B}
\left(
\Map_{\bAlg(\cC)_{A/}}(B,B\oplus M)
\xrightarrow{p_M\circ-}
\Map_{\bAlg(\cC)_{A/}}(B,B)
\right).
\]
The zero derivation is the canonical section
$z_B:B\to B\oplus M$.
\end{defn}

To relate this definition to the usual Leibniz rule, let
$\pi_M:B\oplus M\to M$ be the biproduct projection.
A section $s$ determines an underlying morphism
\[
d_s:=\pi_Ms:B\longrightarrow M.
\]
Projecting the multiplicativity and unit identities for $s$
onto $M$ gives
\[
d_s\mu_B
\simeq
\ell_M(\id_B\otimes d_s)
+
r_M(d_s\otimes\id_B),
\qquad
d_s\eta_B\simeq0,
\]
where
$\ell_M:B\otimes M\to M$ and
$r_M:M\otimes B\to M$ are the bimodule actions.
For a relative derivation, the identification
$sf\simeq z_Bf$ also supplies
\[
d_sf\simeq0.
\]

\begin{prop}[Addition of derivations]
\label{prop:derivations_square_zero}
The spaces $\operatorname{Der}(B,M)$ and
$\operatorname{Der}_A(B,M)$ carry canonical grouplike
$E_\infty$-structures, naturally in the coefficient
bimodule $M$. Their zero elements are the zero sections,
and their addition is induced by addition in $M$.
\end{prop}

\begin{proof}
The split square-zero construction satisfies
\[
(B\oplus M)\times_B(B\oplus N)
\simeq
B\oplus(M\oplus N)
\]
as augmented algebras over $B$.
Consequently, the bimodule addition $M\oplus M\to M$
induces an algebra morphism
\[
(B\oplus M)\times_B(B\oplus M)
\simeq
B\oplus(M\oplus M)
\xrightarrow{\id_B\oplus+}
B\oplus M.
\]
Similarly, negation is induced by
$\id_B\oplus(-\id_M)$, and the identity element is $z_B$.
These operations make $p_M$ a commutative group object
in $\bAlg(\cC)_{/B}$, with the compatibilities inherited
from the additive structure of $M$.

Applying
$\Map_{\bAlg(\cC)_{/B}}(\id_B,-)$
gives the asserted structure on
$\operatorname{Der}(B,M)$.
All these operations preserve the structure map $z_Bf$,
so the same argument in the category of algebras under $A$
and over $B$ applies to $\operatorname{Der}_A(B,M)$.
\end{proof}

\section{Homotopical and homological epimorphisms}

Throughout this section, $\cC$ is monoidal, admits geometric realizations, and tensoring preserves them separately in each variable. We assume that the algebra pushouts under consideration exist. Additivity, preservation of finite coproducts and closedness are imposed only where stated.

\begin{defn} \label{defn:homotopy_epimorphism}
A morphism $f: A \to B$ in $\bAlg(\cC)$ is said to be
\begin{itemize}
\item[-] a \emph{homotopical epimorphism} if the canonical morphism $B \ast_A B \to B$ is an equivalence,
\item[-] a \emph{homological epimorphism} if the canonical morphism $B \otimes_A B \to B$ is an equivalence of $B$-bimodules.
\end{itemize}
\end{defn}

The first condition is equivalent to restriction
\[
\bAlg(\cC)_{B/}\longrightarrow\bAlg(\cC)_{A/}
\]
being fully faithful. Indeed, by the following lemma, restriction $\Map(B,S)\to\Map(A,S)$ is $(-1)$-truncated.
For specified $B$-algebras $R,S$, taking fibers therefore identifies the mapping spaces under $B$ and under $A$.
Conversely, if restriction between undercategories is fully faithful and an $A$-algebra $R$ admits a $B$-structure, then $\Map_{\bAlg(\cC)_{A/}}(B,R)$ is equivalent to the contractible space $\Map_{\bAlg(\cC)_{B/}}(B,R)$.

\begin{lemma}
\label{lemma:homotopical_epi_mapping_spaces}
Let $f:A\to B$ be an associative algebra morphism, and suppose
that the pushout $B\ast_A B$ exists. The following conditions
are equivalent:
\begin{enumerate}
\item The morphism $f$ is a homotopical epimorphism.
\item For every $R\in\bAlg(\cC)$, restriction induces a
$(-1)$-truncated map
\[
f^*:
\Map_{\bAlg(\cC)}(B,R)
\longrightarrow
\Map_{\bAlg(\cC)}(A,R).
\]
\item For every $A$-algebra $R$, the space
\[
\Map_{\bAlg(\cC)_{A/}}(B,R)
\]
is either empty or contractible.
\end{enumerate}
In particular, whenever a specified algebra morphism
$A\to R$ extends along $f$, its space of extensions is
contractible.
\end{lemma}

\begin{proof}
For every algebra $R$, the pushout property gives
\[
\Map_{\bAlg(\cC)}(B\ast_A B,R)
\simeq
\Map_{\bAlg(\cC)}(B,R)
\times_{\Map_{\bAlg(\cC)}(A,R)}
\Map_{\bAlg(\cC)}(B,R).
\]
Under this identification, restriction along the codiagonal
$B\ast_A B\to B$ is the diagonal of $f^*$.

A map of spaces is $(-1)$-truncated precisely when its
diagonal is an equivalence. Thus condition~(2) holds
precisely when restriction along the codiagonal induces
an equivalence on mapping spaces into every algebra $R$.
By the Yoneda lemma, this is equivalent to the codiagonal
itself being an equivalence, which is condition~(1).

Finally, for a specified structure map $\alpha:A\to R$,
there is a canonical equivalence
\[
\Map_{\bAlg(\cC)_{A/}}(B,R)
\simeq
\operatorname{fib}_{\alpha}(f^*).
\]
Consequently, conditions~(2) and~(3) are equivalent.
\end{proof}

\begin{prop} \label{prop:fully_faithful}
Let $f:A\to B$ be a morphism in $\bAlg(\cC)$. The following conditions are equivalent:
\begin{itemize}
\item[-] the restriction of scalars functor $\bMod_B \to \bMod_A$ is fully faithful;
\item[-] the restriction of scalars functor ${}_B\bMod \to {}_A\bMod$ is fully faithful;
\item[-] the restriction of scalars functor ${}_B\bMod_B \to {}_A\bMod_A$ is fully faithful.
\end{itemize}
Equivalently, each of these conditions holds if and only if $f$ is a homological epimorphism.
\end{prop}
\begin{proof}
Restriction is right adjoint to extension of scalars
\[
-\otimes_A B,\qquad B\otimes_A-,\qquad
B\otimes_A(-)\otimes_A B,
\]
respectively. A right adjoint is fully faithful precisely when its counit is an equivalence. For one-sided modules the counit at $B$ is $\mu:B\otimes_A B\to B$, so full faithfulness implies that $\mu$ is an equivalence, also as a bimodule by conservativity.

For bimodules the counit at $B$ is instead the triple multiplication
\[
\mu_3:B\otimes_A B\otimes_A B\longrightarrow B.
\]
The arrow $\mu$ is a retract of $\mu_3$. Explicitly, set
\[
s(b_1\otimes b_2)=b_1\otimes1\otimes b_2,\qquad
r(b_1\otimes b_2\otimes b_3)=b_1b_2\otimes b_3.
\]
These denote coherent bimodule maps obtained from the unit and multiplication. They satisfy $rs\simeq\id$, $\mu_3s\simeq\mu$, and $\mu r\simeq\mu_3$. Thus an equivalence $\mu_3$ implies that $\mu$ is an equivalence.

Conversely, if $\mu$ is an equivalence, the one-sided counits are obtained by tensoring $\mu$ over $B$ with the given module. For $X\in{}_B\bMod_B$, the counit identifies with
\[
(B\otimes_A B)\otimes_B X\otimes_B(B\otimes_A B)
 \xrightarrow{\mu\otimes_B\id_X\otimes_B\mu}X.
\]
All these maps are equivalences. Therefore the three restriction functors are fully faithful.
\end{proof}

We will also use the following simple lemma.

\begin{lemma}
\label{lemma:restriction_tensor_equiv}
Let \(f:A\to B\) be a homological epimorphism in \(\bAlg(\cC)\). Then the restriction of scalars functor
\[
U_f:{}_B\bMod_B \to {}_A\bMod_A
\]
is lax monoidal and the canonical natural transformation
\[
U_f(M)\otimes_A U_f(N) \longrightarrow U_f(M\otimes_B N),
\]
is an equivalence.
\end{lemma}

\begin{rmk}
The lax monoidal unit of $U_f$ is $A\to B$, which need not be an equivalence. This is exactly the situation covered by Lemma~\ref{lemma:lax_ff_alg}.
\end{rmk}

\begin{proof}[Proof of Lemma \ref{lemma:restriction_tensor_equiv}]
The lax monoidal structure is the canonical map from the relative tensor product over \(A\) to the relative tensor product over \(B\).
Assume now that \(\mu_B: B \otimes_A B \to B\) is an equivalence. For \(M,N\in {}_B\bMod_B\), using associativity of the relative tensor product we obtain
\[
U_f(M)\otimes_A U_f(N)
\simeq
M\otimes_B (B\otimes_A B)\otimes_B N.
\]
Under this identification, the lax structure is induced by
\[
\id_M\otimes_B \mu_B \otimes_B \id_N :
M\otimes_B (B\otimes_A B)\otimes_B N \to M\otimes_B B\otimes_B N \simeq M\otimes_B N.
\]
Since \(\mu_B\) is an equivalence, it follows that
\[
U_f(M)\otimes_A U_f(N) \longrightarrow U_f(M\otimes_B N)
\]
is an equivalence.
\end{proof}

We now fix some notation that will be used until the end of this section.
Let \(f:A\to B\) be a morphism in \(\bAlg(\cC)\), and let \(M\in {}_B\bMod_B\).
We denote by
\[
l_{M,B}: B\otimes M \to M,
\qquad
r_{M,B}: M\otimes B \to M
\]
the left and right \(B\)-actions on \(M\). By restriction of scalars along \(f\), we also obtain maps
\[
l_{M,A}: A\otimes M \to M,
\qquad
r_{M,A}: M\otimes A \to M.
\]
The left and right actions retain their specified input orders.

\begin{lemma}\label{lemma:relative_derivation_obstruction}
Assume that $\cC$ is additive and that tensoring preserves finite coproducts. If $f:A\to B$ is homotopical, then for every $B$-bimodule $M$ the space $\operatorname{Der}_A(B,M)$ is contractible. In particular, every relative derivation is homotopic to the zero section as an algebra map under $A$.
\end{lemma}
\begin{proof}
For each $R$ with a fixed $A$-structure, the space
$\Map_{\bAlg(\cC)_{A/}}(B,R)$ is a homotopy fiber of restriction of algebra maps. It is therefore empty or contractible by Lemma \ref{lemma:homotopical_epi_mapping_spaces}. For $R=B\oplus M$ and $R=B$ it is nonempty, containing $z_B$ and $\id_B$, respectively. Hence the fiber defining $\operatorname{Der}_A(B,M)$ is contractible.
\end{proof}

\begin{lemma}\label{lemma:canonical_derivation}
Now suppose also that $\cC$ is additive and that tensoring
preserves finite coproducts separately.
Let $f:A\to B$, and put
\[
M=B\otimes_A B,
\]
with its outer $B$-bimodule structure. Define
\[
\begin{aligned}
\sigma_1:\;
B&\simeq B\otimes_A A
\xrightarrow{\id_B\otimes_A f}B\otimes_A B,\\
\sigma_2:\;
B&\simeq A\otimes_A B
\xrightarrow{f\otimes_A\id_B}B\otimes_A B.
\end{aligned}
\]
There exists an $A$-relative derivation, represented by
an algebra section
\[
s:B\longrightarrow B\oplus M,
\]
whose underlying $A$-bimodule morphism is
\[
\delta:=\pi_Ms
\simeq\sigma_1-\sigma_2:B\longrightarrow M.
\]

If $\delta$ is nullhomotopic in
$\Map_{{}_A\bMod_A}(B,M)$, then the multiplication
\[
\mu_{B/A}:B\otimes_A B\longrightarrow B
\]
is an equivalence.
In particular, this holds if $s$ is homotopic to the
zero section as an algebra morphism under $A$.
\end{lemma}

\begin{proof}
Work in the additive monoidal category
\[
\cD=({}_A\bMod_A,\otimes_A),
\]
whose monoidal unit is $A$. Tensoring in $\cD$ preserves finite coproducts separately,
as follows from its bar construction.
By Proposition~\ref{prop:algebras_bimodules}, the morphism
$f$ determines an associative algebra $B$ in $\cD$.
The object $M$ is a $B$-bimodule in $\cD$, and there is
a canonical morphism
\[
u:A\simeq A\otimes_A A
\xrightarrow{f\otimes_A f}M.
\]
Writing $\ell_M$ and $r_M$ for its left and right
$B$-actions, respectively, we have
\[
\sigma_1 \simeq \ell_M(\id_B\otimes_A u),
\qquad
\sigma_2 \simeq r_M(u\otimes_A\id_B).
\]

Form the split square-zero algebra
\[
S=B\oplus M
\]
in $\cD$, with augmentation $p:S\to B$ and zero
section $z_B:B\to S$.
Let $v:A\to S$ be the composite of $u$ with the inclusion
of $M$ into $S$.
Since $M$ is square-zero, the elements
\[
t=\eta_S-v,
\qquad
t^{-1}=\eta_S+v
\]
of the multiplicative monoid $\Map_\cD(A,S)$
are mutually inverse. Conjugation $c_t(a) = t a t^{-1}$ therefore gives an
algebra automorphism
\[
c_t:S\longrightarrow S.
\]
Indeed, the free-right-module adjunction gives
\[
\Map_{\bMod_S(\cD)}(X\otimes_A S,S)
\simeq
\Map_\cD(X,S).
\]
Thus the enriched endomorphism object of the regular
right $S$-module is represented by $S$, with composition
corresponding to multiplication.
Left multiplication by $t$ is an automorphism of that
module. Transporting endomorphisms along this automorphism
gives $c_t$, including its algebra coherences.
See the endomorphism-algebra construction in
\cite[\S4.7.1]{HA}.

Set
\[
s:=c_t\circ z_B:B\longrightarrow S.
\]
Since $p(t)=\eta_B$, conjugation becomes the identity
after applying $p$. Hence
\[
ps\simeq\id_B.
\]
The morphism $s$ is therefore an algebra section in $\cD$.
Under the equivalence of
Proposition~\ref{prop:algebras_bimodules}, it determines
an $A$-relative derivation in $\cC$.

Expanding the conjugation using the square-zero
multiplication gives
\[
\pi_Ms
\simeq
\ell_M(\id_B\otimes_A u)
-
r_M(u\otimes_A\id_B)
\simeq
\sigma_1-\sigma_2.
\]
This identifies its underlying morphism with $\delta$.

Now suppose that $\delta\simeq0$ in $\cD$.
Additivity gives $\sigma_1\simeq\sigma_2$ there.
Writing $\mu=\mu_{B/A}$, unitality gives
\[
\mu\sigma_1\simeq\id_B.
\]
Furthermore, associativity and the homotopy
$\sigma_1\simeq\sigma_2$ give
\[
\begin{aligned}
\sigma_1\mu
&\simeq
(\mu\otimes_A\id_B)
(\id_B\otimes_A\sigma_1)\\
&\simeq
(\mu\otimes_A\id_B)
(\id_B\otimes_A\sigma_2)\\
&\simeq
\id_{B\otimes_A B}.
\end{aligned}
\]
These are homotopies of $A$-bimodule morphisms.
Consequently, $\sigma_1$ is an inverse of $\mu$ in $\cD$.
Conservativity of the forgetful functors then implies
that $\mu$ is an equivalence of $B$-bimodules.

Finally, a homotopy $s\simeq z_B$ in
$\bAlg(\cC)_{A/}$ induces a homotopy of the corresponding
algebra morphisms in $\cD$. Forgetting their algebra
structures and composing with $\pi_M$ gives
\[
\delta=\pi_Ms\simeq\pi_Mz_B=0
\]
in $\Map_{{}_A\bMod_A}(B,M)$, proving the last assertion.
\end{proof}

The following is our main result.

\begin{thm}\label{thm:main}
Let $\cC$ be a monoidal $\infty$-category admitting geometric realizations, and suppose that tensoring preserves geometric realizations separately in each variable. Suppose also that the algebra pushout $B\ast_A B$ exists for the morphism $f:A\to B$ under consideration.
\begin{enumerate}
\item Every homological epimorphism $f$ is homotopical.
\item If $\cC$ is additive and tensoring preserves finite coproducts separately in each variable, then $f$ is homotopical if and only if it is homological.
\end{enumerate}
\end{thm}
\begin{proof}
If $f$ is homological, restriction
\[
U_f:{}_B\bMod_B\longrightarrow{}_A\bMod_A
\]
is fully faithful by Proposition~\ref{prop:fully_faithful}. By Lemma~\ref{lemma:restriction_tensor_equiv} it satisfies the hypothesis of Lemma~\ref{lemma:lax_ff_alg} and therefore it gives a fully faithful functor on associative algebras. Proposition~\ref{prop:algebras_bimodules} identifies this functor with
\[
\bAlg(\cC)_{B/}\longrightarrow\bAlg(\cC)_{A/},
\]
hence $f$ is homotopical.

For the converse, suppose that $\cC$ is additive, $\otimes$ is additive and assume that $f$ is homotopical. Apply Lemma~\ref{lemma:relative_derivation_obstruction}
to the relative derivation constructed in Lemma~\ref{lemma:canonical_derivation}.
The corresponding section
\[
s:B\longrightarrow B\oplus(B\otimes_A B)
\]
is homotopic to the zero section $z_B$ in $\bAlg(\cC)_{A/}$.
The forgetful functor from algebras under $A$ to
$A$-bimodules sends the homotopy $s\simeq z_B$ to a
homotopy of their underlying $A$-bimodule morphisms.
Composing with the $A$-bimodule projection
\[
\pi_M:B\oplus M\longrightarrow M,
\qquad M=B\otimes_A B,
\]
therefore gives
\[
\delta=\pi_Ms\simeq\pi_Mz_B=0
\quad\text{in }\Map_{{}_A\bMod_A}(B,M).
\]
Lemma~\ref{lemma:canonical_derivation} now implies that relative multiplication is an equivalence.
Thus $f$ is homological.
\end{proof}

The hypothesis of additivity on $\cC$ for the implication that homotopical epimorphism implies homological epimorphism is necessary. Remark~\ref{rmk:additivity_necessary} gives a counterexample in the $\infty$-category of spaces. Appendix~\ref{app:closed_proof} gives another proof of Theorem \ref{thm:main} when $\cC$ is additive and the functors $-\otimes X$ admit right adjoints.

We conclude this section with a few consequences of the main theorem.

\begin{cor}\label{cor:map_on_centers}
A homological epimorphism $f:A\to B$ induces a morphism of composition monoids
\[
\mathcal Z(A)\longrightarrow\mathcal Z(B).
\]
In particular, this holds for homotopical epimorphisms under the hypotheses of Theorem~\ref{thm:main}(2).
\end{cor}
\begin{proof}
The functor $L_f(M)=B\otimes_A M\otimes_A B$ from $A$-bimodules to $B$-bimodules sends $A$ to $B\otimes_A B\simeq B$. Applying it to endomorphisms of $A$ gives the map, preserving coherent composition.
\end{proof}

For the following enriched statement, assume that $\cC$
is presentable symmetric monoidal and that its tensor
product preserves small colimits separately in each
variable. These hypotheses ensure the existence of the
center objects with their canonical $E_2$-structures.

\begin{cor}\label{cor:map_on_centers_enriched}
Assume that $\cC$ is presentable symmetric monoidal and that tensoring preserves small colimits separately. A homological epimorphism $f:A\to B$ induces a morphism of $E_2$-algebras
\[
\underline{\mathcal Z}(A)\longrightarrow\underline{\mathcal Z}(B).
\]
The induced map on center spaces, obtained by applying
$\Map_\cC(1_\cC,-)$, agrees with the map constructed in
Corollary~\ref{cor:map_on_centers}. Thus the same conclusion holds for homotopical epimorphisms when $\cC$ is additive.
\end{cor}
\begin{proof}
We use left modules, so that Eilenberg--Watts has the usual composition order
\[
({}_A\bMod_A,\otimes_A)
 \simeq\operatorname{Fun}_{\cC}^{L}({}_A\bMod,{}_A\bMod),
\qquad M\longmapsto M\otimes_A-.
\]
See \cite[Theorem~4.8.4.1]{HA}. The unit corresponds to the identity functor, and its enriched endomorphisms give the $E_2$-center.

Restriction $U:{}_B\bMod\to{}_A\bMod$ is fully faithful and $\cC$-linear. Its enriched full faithfulness follows from full faithfulness on mapping spaces and by testing the enriched comparison against every $X\in\cC$ and using $U(X\otimes M)\simeq X\otimes U(M)$. The central action of $\underline{\mathcal Z}(A)$ on the identity restricts to the full subcategory identified with the category of left $B$-modules. Full faithfulness lifts this action, including its operadic coherences, to the identity functor there. The universal property of the center gives an $E_2$-map into $\underline{\mathcal Z}(B)$. This construction is enriched: for every parameter $X\in\cC$, a transformation $X\otimes-\to\id$ restricts along $U$, and enriched full faithfulness lifts it. Thus it defines a morphism of objects in $\cC$, not merely a map on points.

An enriched endomorphism $\alpha:X\otimes A\to A$ of the unit acts on the free $B$-module by $\alpha\otimes_A B$. Under the extension--restriction adjunction, the induced $B$-bimodule endomorphism is
\[
X\otimes B\simeq B\otimes_A(X\otimes A)\otimes_A B
 \xrightarrow{\id_B\otimes_A\alpha\otimes_A\id_B}
 B\otimes_A B\simeq B.
\]
The identifications use the counit equivalences and are natural in $X$ and $\alpha$. This identifies the map with the enriched version of $L_f$ and proves the comparison on center spaces.
\end{proof}

\begin{cor}\label{cor:epi_from_commutative_is_commutative}
Assume that $\cC$ is presentable symmetric monoidal and tensoring preserves small colimits separately. Let $f:A\to B$ be a homological epimorphism, and fix a commutative structure on $A$ refining its associative structure. Then $B$ admits a commutative refinement such that $f$ lifts to a morphism of commutative algebras. The space of such refinements and lifts, with their identifications with the given associative morphism,
is contractible.
\end{cor}
\begin{proof}
The central self-action of $A$ gives a morphism $A\to\underline{\mathcal Z}(A)$ of $E_2$-algebras. Composing with Corollary~\ref{cor:map_on_centers_enriched} gives
\[
A\longrightarrow\underline{\mathcal Z}(B)
\]
whose evaluation is $f$. Proposition~\ref{prop:central_algebra_definition} therefore gives a coherent central structure on $f$, making $B$ an associative algebra in $(\bMod_A,\otimes_A)$. Its multiplication $\mu$ is an equivalence, and unitality makes $\id_B\otimes_A f$ its inverse. Thus $(B,f)$ is an idempotent pointed object. Proposition~4.8.2.9 of \cite{HA} supplies its commutative algebra structure, uniquely up to contractible choices. Forgetting the commutative structure recovers the given
morphism $f:A\to B$, up to equivalence in
$\bAlg(\cC)_{A/}$. The restriction of scalars $\bMod_B\longrightarrow\bMod_A$ satisfies the hypotheses of Lemma~\ref{lemma:lax_ff_alg} and therefore induces a fully faithful inclusion of $\bAlg(\bMod_B)$ into $\bAlg(\bMod_A)$.
Under the equivalence
\[
\bC\bAlg(\bMod_A)\simeq\bC\bAlg(\cC)_{A/},
\]
the commutative algebra constructed above corresponds to
a commutative algebra structure on $B$ together with a
morphism of commutative algebras $A\to B$ whose underlying
associative morphism is equivalent to $f$. Any other refinement yields the same pointed $A$-module: its action is restriction of the fixed regular $B$-module along the fixed associative map $f$. The full faithfulness in Proposition~4.8.2.9 of \cite{HA}, together with the uniqueness statements above shows that the space of commutative refinements, together with an identification of their underlying associative morphism with the given $f:A\to B$, is contractible.
\end{proof}

\begin{rmk}
Corollary~\ref{cor:epi_from_commutative_is_commutative} extends a classical theorem of Silver. If $f:A\to B$ is an epimorphism in the category of unital associative
rings and $A$ is commutative, then $B$ is commutative \cite[Corollary~1.2]{SIL}.
\end{rmk}

\begin{rmk}\label{rmk:hochschild_gerstenhaber}
Under the hypotheses of Corollary~\ref{cor:map_on_centers_enriched}, the enriched center is the Hochschild cochain object
\[
\underline{\mathcal Z}(A)
 \simeq\underline{\operatorname{HH}}^\bullet(A)
 =\underline{\Hom}_{{}_A\bMod_A}(A,A).
\]
Corollary~\ref{cor:map_on_centers_enriched} shows that homological epimorphisms induce morphisms $\underline{\operatorname{HH}}^\bullet(A) \to \underline{\operatorname{HH}}^\bullet(B)$ of $E_2$-algebras.
\end{rmk}

\section{Examples and applications}  \label{sec:applications}

We conclude with some consequences and examples of Theorem \ref{thm:main}. Besides recovering the equivalence between homological and homotopical epimorphisms for differential graded algebras, the theorem gives a simple characterization in terms of the associative cotangent complex and provides a connection with the notion of localization appearing in derived algebraic and analytic geometry. 

\subsection{Relative differentials and the cotangent complex} 
\addtocontents{toc}{\protect\setcounter{tocdepth}{1}}

Assume in this subsection that $\cC$ is symmetric monoidal, stable and presentable and that tensoring preserves small colimits in each variable. For $f:A\to B$, the associative relative cotangent complex is
\[
\mathbb L^{\mathrm{nc}}_{B/A}
 \simeq\operatorname{fib}\left(B\otimes_A B
                 \xrightarrow{\mu_{B/A}}B\right)
\]
as a $B$-bimodule. This is the relative associative form of \cite[Theorem~7.3.5.1]{HA}. The canonical relative derivation has its universal coherent lift
\[
d_{B/A}:B\longrightarrow\mathbb L^{\mathrm{nc}}_{B/A},
\qquad
\operatorname{Der}_A(B,M)
 \simeq\Map_{{}_B\bMod_B}(\mathbb L^{\mathrm{nc}}_{B/A},M).
\]
\begin{cor}\label{cor:cotangent_epi}
Under the hypotheses fixed on $\cC$, the following are equivalent:
\begin{enumerate}
\item $f$ is homotopical;
\item $f$ is homological;
\item $\mathbb L^{\mathrm{nc}}_{B/A}\simeq0$.
\end{enumerate}
\end{cor}
\begin{proof}
The first equivalence is Theorem~\ref{thm:main}. Stability makes a morphism an equivalence precisely when its fiber is zero.
\end{proof}

For differential graded algebras, this recovers the derived module of relative differentials used by Chuang and Lazarev \cite[\S3]{ChuangLazarev}. In particular, Theorem~\ref{thm:main} specializes to their equivalence between homological and homotopical epimorphisms of dg-algebras \cite[Theorem~4.4]{ChuangLazarev}.

\begin{rmk}
The associative cotangent complex
$\mathbb L^{\mathrm{nc}}_{B/A}$ represents associative derivations with values in arbitrary $B$-bimodules. It should be distinguished from Illusie's commutative cotangent complex $\mathbb L^{\mathrm{comm}}$, which govern commutative deformation theory. Vanishing of the commutative cotangent complex does not imply that a morphism is a homotopical epimorphism. For example, let $k$ be a field and $B=k\times k$. The diagonal map $k\to B$ is finite \'etale, hence $\mathbb L^{\mathrm{comm}}_{B/k}\simeq0$, whereas
\[
\mathbb L^{\mathrm{nc}}_{B/k}
\simeq
\ker\bigl(B\otimes_k B\xrightarrow{\mu}B\bigr)
\neq0.
\]
Indeed, writing $e_1,e_2$ for the standard idempotents of $B$, this kernel is spanned by $e_1\otimes e_2$ and $e_2\otimes e_1$. Thus $k\to B$ is not homological and, by Theorem~\ref{thm:main}, is not homotopical.
\end{rmk}

\subsection{Unstable homotopical epimorphisms}
\addtocontents{toc}{\protect\setcounter{tocdepth}{1}}
Theorem \ref{thm:main} extends the equivalence between homological and
homotopical epimorphisms proved by Chuang and Lazarev for differential
graded algebras. Their result is proved for the stable derived category of chain complexes. In
\cite[Remark 4.2]{ChuangLazarev}, they ask whether the equivalence is a
general categorical phenomenon rather than a peculiarity of dg-algebras.

Theorem~\ref{thm:main} shows that the stability of $\cC$ is not required: additivity together with preservation of finite coproducts by tensoring suffices for the converse in the abstract setting. In particular, our result applies to the following unstable example.

\begin{exa}
Let $\cC$ be simplicial abelian groups with weak equivalences inverted and the derived tensor product. Dold--Kan identifies its underlying $\infty$-category with $D^{\leq0}(\mathbb Z)$. It is additive and presentable, tensoring preserves small colimits separately, and it is not stable. Consequently, for a map of simplicial associative rings in the derived $\infty$-category,
\[
B\ast_A B\simeq B
\quad\Longleftrightarrow\quad
B\otimes_A B\simeq B.
\]
This connective comparison is used (and proved in this specific case) in \cite{BambozziCapoferriMurro}.
\end{exa}

It is worth observing that only one implication requires additivity.
The proof that
\[
\text{homological epimorphism}
\Longrightarrow
\text{homotopical epimorphism}
\]
uses no additive structure. On the other hand, the converse can fail in a non-additive setting.

\begin{rmk}\label{rmk:additivity_necessary}
Consider the $\infty$-category of spaces $\cS$ with cartesian tensor product and a nontrivial acyclic discrete group $G$, regarded as a grouplike $E_1$-space. The map $G\to \star$ is not homological, since
\[
\star \times_G \star \simeq BG,
\]
whose fundamental group is $G$.
It is nevertheless homotopical. The pushout  $\star \ast_G \star$ in $E_1$-spaces is grouplike: $\pi_0$ preserves colimits into ordinary monoids, and this pushout of monoids is trivial. Thus it is also the pushout in grouplike $E_1$-spaces. Their equivalence with pointed connected spaces \cite[\S5.2.6]{HA} gives
\[
\star \ast_G \star \simeq\Omega\Sigma BG.
\]
The suspension of the connected acyclic CW-space $BG$ is simply connected and acyclic, hence contractible by the Hurewicz and Whitehead theorem. The codiagonal is therefore an equivalence.
\end{rmk}

\subsection{Comodules and localization}
\addtocontents{toc}{\protect\setcounter{tocdepth}{1}}
Let $k$ be a field, let $G$ be a nonabelian discrete group, and equip $H=k[G]$ with its usual Hopf algebra structure
\[
\Delta(g)=g\otimes g,\qquad
\epsilon(g)=1,\qquad S(g)=g^{-1}.
\]
Write $\operatorname{Comod}_H$ for the abelian category of right $H$-comodules and set
\[
\cC=D^{\leq0}(\operatorname{Comod}_H).
\]
Its tensor product is induced by tensor product over $k$ with coaction
\[
\rho(x\otimes y)
=\sum x_{(0)}\otimes y_{(0)}\otimes x_{(1)}y_{(1)}.
\]
The unit is $k$ with its trivial coaction. A right $H$-comodule is equivalently described as a $G$-graded vector space
\[
V=\bigoplus_{g\in G}V_g,\qquad
V_g=\{v\mid\rho(v)=v\otimes g\}.
\]
Accordingly, $\cC \simeq \operatorname{Fun}(G_{\mathrm{disc}},D^{\leq0}(k))$, with $(X \otimes Y)_g = \bigoplus_{ab=g} X_a \otimes_k^{\mathbb L} Y_b$. The category $\cC$ is additive and presentable, and its tensor product preserves small colimits separately.
It is not stable and admits no braiding: if $k_g$ is the one-dimensional comodule with coaction $1\mapsto1\otimes g$, then
\[
k_g\otimes k_h\simeq k_{gh},\qquad
k_h\otimes k_g\simeq k_{hg},
\]
and these comodules are not equivalent when $gh\neq hg$.
Consider the $H$-comodule algebras
\[
A=H\otimes_k k[t],\qquad
B=H\otimes_k k[t,t^{-1}],
\]
concentrated in degree zero, with $t$ central and coinvariant
\[
\rho_A(gt^n)=gt^n\otimes g,\qquad
\rho_B(gt^n)=gt^n\otimes g.
\]
As a right $A$-module in $\cC$,
\[
B\simeq\operatorname*{colim}
\bigl(A\xrightarrow{t}A\xrightarrow{t}A\xrightarrow{t}\cdots\bigr),
\]
where the map from stage $n$ is $a\mapsto t^{-n}a$, starting with $n=0$. Hence
\[
B\otimes_A^{\mathbb L}B
\simeq\operatorname*{colim}
\bigl(B\xrightarrow{t}B\xrightarrow{t}B\xrightarrow{t}\cdots\bigr)
\xrightarrow{\;\mu_{B/A}\;}B.
\]
The map from stage $n$ is $b\mapsto t^{-n}b$ and is an equivalence. Thus $A\to B$ is homological and, by Theorem~\ref{thm:main}, homotopical. The same construction applies in $D(\operatorname{Comod}_H)$.

\begin{rmk}
\label{rmk:comodule_localization_geometry}
More generally, if $T$ is an ordinary $H$-comodule algebra
and $s\in T$ is central and coinvariant, the same telescope argument gives
\[
T[s^{-1}]\otimes_T^{\mathbb L}T[s^{-1}]
\xrightarrow{\;\simeq\;}T[s^{-1}].
\]
Its coaction extends by $\rho(s^{-1})=s^{-1}\otimes1$.
Thus central coinvariant localizations are homotopical
epimorphisms of associative algebras in $\cC$. Indeed, for a commutative $k$-algebra $R$, put $A_R=R\otimes_kH$, with coaction $\id_R\otimes\Delta$ and $A_R^{\operatorname{co}H}=R$.
Coinvariants induce an equivalence
\[
\bMod_{A_R}(\cC)\simeq D^{\leq0}(R),\qquad
M\longmapsto M^{\operatorname{co}H},
\]
whose inverse is $N\mapsto N\otimes_R A_R$.
Indeed, on ordinary compatible right modules and comodules,
the action map
\[
M^{\operatorname{co}H}\otimes_k H\longrightarrow M,
\qquad m\otimes g\longmapsto m(1\otimes g),
\]
is an isomorphism: its inverse sends a homogeneous element
$m_g$ to $m_g(1\otimes g^{-1})\otimes g$.
The equivalence of ordinary module categories is a special case of Dade's theorem \cite[Theorem~2.8]{DadeGraded}:
the algebra $A_R$ is strongly $G$-graded, with components $(A_R)_g=R(1\otimes g)$ and identity component $R$, and coinvariants correspond to the identity component
of a graded module. Since this equivalence is exact, it induces the equivalence of derived categories.
\end{rmk}

\subsection{Geometric localizations}
\addtocontents{toc}{\protect\setcounter{tocdepth}{1}}
For the discussion of commutative algebras in this subsection,
the ambient monoidal categories are assumed symmetric.

One of the basic motivations for homological epimorphisms comes from algebraic geometry. For a morphism of commutative rings 
\[ f:A\longrightarrow B, \] 
the condition that \(f\) is an epimorphism is equivalent to the multiplication morphism 
\[ B\otimes_A B\longrightarrow B \] 
being an isomorphism. If \(f\) is moreover flat and of finite presentation, this is precisely the algebraic condition corresponding to an open immersion 
\[ \Spec(B)\longrightarrow\Spec(A). \] 
Thus tensor-idempotent morphisms may be regarded as an algebraic generalization of affine open immersions, once the finiteness and flatness conditions are dropped. The same point of view plays an important role in derived algebraic geometry. In the framework of To\"en and Vezzosi, a morphism $f:A\to B$ of commutative algebras defines a formal Zariski open immersion $\Spec(B)\to\Spec(A)$ when the derived extension of scalars
\[
B\otimes_A^{\mathbb L}-:\bMod_A\longrightarrow\bMod_B
\]
preserves finite homotopy limits and restriction of scalars
\[
f_*:\bMod_B\longrightarrow\bMod_A
\]
is fully faithful \cite[Definition~1.2.6.1(2)--(3)]{HAGII}. The latter condition is equivalent to the multiplication map $B\otimes_A^{\mathbb L}B\to B$ being an equivalence \cite[Proposition~1.2.6.4(1)]{HAGII}. If the ambient category is stable, extension of scalars automatically preserves finite homotopy limits. Thus, in that setting, formal Zariski open immersions correspond precisely to homological epimorphisms of commutative algebras \cite[Corollary~1.2.6.6]{HAGII}. In particular, the identity 
\[ B\otimes_A^{\mathbb L}B\simeq B \] 
is one of the fundamental algebraic manifestations of the fact that \(\Spec(B)\to\Spec(A)\) behaves as a monomorphism of derived affine spaces. This description has a natural noncommutative counterpart. For an associative algebra \(A\), a homological epimorphism 
\[ A\longrightarrow B \] 
is equivalently a morphism for which restriction of scalars 
\[ \bMod_B\longrightarrow\bMod_A \] 
is fully faithful. Hence the category of \(B\)-modules is identified with a full reflective subcategory of the category of \(A\)-modules. This provides a natural replacement, in the absence of an underlying space of prime ideals, for the idea of restricting quasi-coherent modules to an open subspace. This perspective is closely related to the classical theory of noncommutative localization, including Ore and universal localizations, and to their interpretation in terms of localizations of module categories; see for example \cite{SkodaNCLocalization} and the references therein. In particular, stably flat universal localizations provide
standard examples of homological epimorphisms. A geometric use of this idea was developed in \cite{BambozziCapoferriMurro}. There, homotopical epimorphisms of connective derived rings are taken as the basic class of noncommutative localizations. If 
\[ X=\Spec(A) \] 
denotes the corresponding formal noncommutative affine object, a homotopical epimorphism 
\[ A\longrightarrow B \] 
is regarded contravariantly as a homotopy monomorphism 
\[ \Spec(B)\longrightarrow\Spec(A). \]
Families of such morphisms are then used to define a formal homotopy Zariski topology and, ultimately, a noncommutative spectrum. For commutative algebras this construction can be compared with the usual Grothendieck spectrum, while in the genuinely noncommutative case it provides a geometric interpretation of suitable localizations of derived module categories. Theorem \ref{thm:main} shows that, in an additive monoidal $\infty$-category, the two possible descriptions of these localizations coincide: 
\[ B\otimes_A B\simeq B \qquad\Longleftrightarrow\qquad B\ast_A B\simeq B. \] 
Thus the module-theoretic condition of being a homological epimorphism is equivalent to the intrinsic algebraic condition that the codiagonal 
\[ B\ast_A B\longrightarrow B \] 
be an equivalence. In particular, the class of noncommutative open localizations used in \cite{BambozziCapoferriMurro} can equivalently be described in terms of the corresponding localization of module categories. Theorem \ref{thm:main} shows moreover that this identification is not specific to differential graded rings, but is a general feature of additive monoidal $\infty$-categories that satisfy the mild hypotheses of the theorem.

\subsection{Bornological setting}
\addtocontents{toc}{\protect\setcounter{tocdepth}{1}}
Let $k$ be a complete valued field and let $\bCBorn_k$ denote the category of complete bornological $k$-vector spaces of convex type. Bornological vector spaces provide a convenient functional-analytic setting in which boundedness, rather than continuity with respect to a fixed topology, is taken as the basic notion. We refer to \cite{MeyerBornological} for a comparison between bornological and topological functional analysis and to \cite{BambozziBenBassatDagger} for the use of complete bornological spaces in analytic geometry. The category $\bCBorn_k$ is additive and quasi-abelian and is equipped with a closed symmetric monoidal structure whose tensor product is the completed projective tensor product 
\[ E\widehat\otimes_k F. \] 
The corresponding internal Hom is the space of bounded linear maps endowed with the equibounded bornology. Moreover, $\bCBorn_k$ admits all limits and colimits and has enough flat projective objects; see \cite[\S3.3]{BambozziBenBassatDagger}. 

The homological algebra of quasi-abelian categories is developed in \cite{SchneidersQA,ProsmansDerived}. Derived bornological and Ind-Banach module categories provide the framework for the analytic localizations in \cite{BambozziBenBassatDagger,BenBassatKremnizerRAG,BambozziBenBassatKremnizerStein,BenBassatMukherjee}. Their relevant algebraic condition is
\[
B\widehat\otimes_A^{\mathbb L}B\simeq B.
\]
The $\infty$-category $D(\bCBorn_k)$ satisfies the hypotheses of Theorem~\ref{thm:main}, and its tensor product preserves finite coproducts. Hence, the theorem identifies this condition with
\[
B\ast_A^\L B\simeq B.
\]
The same applies to a connective subcategory $D^{\le 0}(\bCBorn_k)$. This is one of the key inputs needed to transport to  the analytic setting the ideas developed in \cite{BambozziCapoferriMurro}. 

The geometric meaning of homological epimorphism has been established in several classes of commutative analytic algebras and turns out to be very similar to the work of To\"en and Vezzosi in the algebraic framework.
For a morphism $A\to B$ of affinoid algebras over a complete
non-archimedean field, Ben-Bassat and Kremnizer prove that
\[
B\widehat\otimes_A^{\mathbb L}B\simeq B
\]
holds if and only if the induced morphism of Berkovich spectra
$\mathcal M(B)\to\mathcal M(A)$ is an affinoid domain embedding
\cite[Theorems~5.16 and~5.31]{BenBassatKremnizerRAG}.
They also show that the resulting formal homotopical Zariski
topology agrees with the weak $G$-topology
\cite[Theorem~5.37]{BenBassatKremnizerRAG}.
Here affinoid domains are understood in the sense of admissible open subset for the $G$-topology, they need not be open subsets of the underlying Berkovich topological space.

The paper \cite{BambozziBenBassatDagger} establishes the analogous theorem for dagger affinoid domains, using the completed bornological tensor product, over both archimedean and non-archimedean valued fields \cite[Theorems~5.7 and~5.11]{BambozziBenBassatDagger}. The corresponding comparison of covering families recovers the weak $G$-topology of dagger affinoid spaces \cite[Theorem~5.15]{BambozziBenBassatDagger}.

For the Stein spaces and algebras considered in \cite{BambozziBenBassatKremnizerStein}, the same condition characterizes open immersions: a morphism $Y\to X$ is an open immersion if and only if the induced morphism $\mathcal O(X)\to\mathcal O(Y)$ is a homotopical epimorphism in the bornological setting \cite[Theorems~5.5 and~5.7]{BambozziBenBassatKremnizerStein}. The comparison of coverings is given in \cite[Corollary~5.14]{BambozziBenBassatKremnizerStein}. In particular, over $\mathbb C$ these results recover the open subsets of the usual analytic topology of complex Stein spaces.

The complex analytic case has earlier antecedents in Taylor's theory of localizations of topological algebras \cite[Proposition~4.3]{TaylorFunctionalCalculus} and Pirkovskii's results for Stein manifolds \cite[Theorem~3.1]{PirkovskiiSteinHomology}.
In the setting of Fr\'echet algebras, Aristov and Pirkovskii characterize open embeddings of Stein spaces by weak homological epimorphisms, or equivalently by $1$-pseudoflat epimorphisms, in their terminology \cite[Theorem~4.2]{AristovPirkovskiiOpen}.
They also prove a corresponding characterization for smooth manifolds using their algebras of smooth functions \cite[Theorem~5.3]{AristovPirkovskiiOpen}.

Beyond this classical setting, more advanced results in this direction have been obtained in \cite[\S\S4--5]{BambozziKremnizerSheafyness} and \cite[Theorem~5.8 and \S6]{BenBassatMukherjee}.

We also mention that in \cite[\S2]{BambozziMiharaHomotopy} it is shown that homotopical epimorphisms between commutative $C^\ast$-algebras correspond to closed embeddings of their associated Gelfand spectra.

Homological epimorphisms also play a crucial role in \cite{BambozziChiarellottoVanniTempered}. There, the lattice of homological epimorphisms from a commutative ind-Banach algebra is used to construct its derived analytic spectrum, together with its structure sheaf \cite[\S5]{BambozziChiarellottoVanniTempered}.
In particular, the homological epimorphism $R[t]\to R[[t]]_{\mathrm{temp}}$ realizes the algebra of tempered power series as functions on an open subset of the analytic affine line. This construction extends to tempered tubular neighborhoods of varieties in smooth formal schemes \cite[Propositions~5.27 and~8.12]{BambozziChiarellottoVanniTempered}.
The Hodge-completed derived de Rham cohomology of these tubes then defines tempered convergent cohomology, a tempered analogue of convergent rigid cohomology \cite[\S9]{BambozziChiarellottoVanniTempered}.

\subsection{Condensed setting}
\addtocontents{toc}{\protect\setcounter{tocdepth}{1}}
Condensed mathematics supplies further examples. Fix cardinal bounds for the profinite site, and let $\CondAb$ denote the resulting category of condensed abelian groups. It is an abelian category with compact projective generators. Its derived category $D(\CondAb)$ is stable and presentable with its derived symmetric monoidal tensor product, and $D^{\leq0}(\CondAb)$ is additive but not stable. The derived tensor product preserves colimits separately and preserves connective objects. Thus both categories satisfy the hypothesis of Theorem \ref{thm:main}; see \cite{ScholzeCondensed} for more information on this framework.

Solidification gives a symmetric monoidal localization and a corresponding derived solid tensor product
\[
-\otimes^{\mathbb L,\solid}-.
\]
For analytic rings, it is preferable to work directly with the stable derived category $D(\mathcal R)$ of complete modules. For a commutative analytic base $\mathcal R$, this has the analytic tensor product, and derived completion is left adjoint to its full embedding in condensed modules; see \cite{ScholzeCondensed,ScholzeAnalytic}. 

For associative algebra objects $P\to Q$ in a fixed analytic module category, we obtain
\[
Q\otimes_P^{\mathbb L,\mathcal R}Q\simeq Q
\quad\Longleftrightarrow\quad
Q\ast_P Q\simeq Q.
\]
The same statement holds in its connective part. This provides the algebraic input for condensed or solid versions of the noncommutative constructions of \cite{BambozziCapoferriMurro}. 

Homological epimorphisms have also been considered in the condensed approach to analytic geometry developed by Clausen and Scholze, frequently under the terminology of \emph{idempotent algebras}. Indeed, for
a morphism $A\to B$ of commutative algebras in a fixed symmetric monoidal derived category, the condition
\[
B\otimes_A^{\mathbb L}B\xrightarrow{\;\simeq\;}B
\]
says precisely that $B$ is an idempotent algebra in the category
of $A$-modules. Equivalently, restriction of scalars is fully
faithful, and $M\mapsto M\otimes_A^{\mathbb L}B$ is a
localization. These statements apply, in particular, to the
derived categories of solid and liquid modules, using their
respective tensor products.

An early example in the solid theory is the algebra
$\mathbb Z((T^{-1}))$ of functions near the boundary of
$\operatorname{Spec}\mathbb Z[T]$. Its idempotence in the
appropriate category of solid $\mathbb Z[T]$-modules is used
in the construction of compactly supported cohomology and
duality; see
\cite[Lecture~VIII, especially Observation~8.7]{ScholzeCondensed}.
The geometric role of localization is developed further in
\cite[Lecture~X]{ScholzeCondensed}.

For general analytic rings, Scholze defines a localization
$\mathcal A\to\mathcal B$ by the full faithfulness of restriction
on connective complete modules
\cite[Definition~12.16]{ScholzeAnalytic}. Among steady
morphisms, this is equivalent to
\[
\mathcal B\otimes_{\mathcal A}\mathcal B\simeq\mathcal B,
\]
where the tensor product is formed in the category of analytic
rings \cite[Exercise~12.17]{ScholzeAnalytic}. Steadiness expresses
compatibility of extension and restriction of scalars with
base change. Finite conservative families of steady
localizations satisfy descent for complete modules
\cite[Proposition~12.18]{ScholzeAnalytic}.

In non-archimedean geometry, Andreychev proves that an affinoid
open immersion into the adic spectrum of a sheafy analytic
complete Huber pair induces a steady localization of the
associated analytic rings
\cite[Proposition~4.12]{AndreychevAdic}. These localizations
provide the descent formalism for derived complete modules
on analytic adic spaces
\cite[Theorem~4.1]{AndreychevAdic}.

In the liquid theory, Clausen and Scholze use idempotent
algebras to construct a locale together with a sheaf of
module categories
\cite[Construction~5.2 and subsequent discussion]
{ClausenScholzeComplex}. Their computation of the idempotence
of algebras of overconvergent holomorphic functions supplies
the analytic input for constructing the holomorphic
structure sheaf
\cite[Proposition~5.6]{ClausenScholzeComplex}.

\subsection{Well-copoweredness}
\addtocontents{toc}{\protect\setcounter{tocdepth}{1}}
Assume here that $\cC$ is a stable presentable monoidal $\infty$-category, tensoring preserves small colimits separately, and $\Ho(\bMod_A)$ is compactly generated for every $A$ under consideration. If $f:A\to B$ is homological, fully faithful restriction identifies $\bMod_B$ with a reflective subcategory of $\bMod_A$. Its localization is
\[
L_B(M)=U_f(M\otimes_A B),\qquad
L_B^2\simeq L_B,
\]
and its kernel is
\[
\mathcal K_B=\{M\mid M\otimes_A B\simeq0\}.
\]
The localization is exact and preserves arbitrary coproducts; its local objects are also closed under coproducts. Thus $\Ho(\mathcal K_B)$ is a smashing subcategory. Smashing subcategories of a compactly generated triangulated category are classified by suitable ideals of morphisms in its essentially small subcategory of compact objects \cite[Corollary 12.5]{KrauseSmashing}. In particular, they form a set.

The kernel $\mathcal K_B$ determines the full subcategory
of local objects as its right orthogonal, and hence
determines the reflective localization $L_B$ and its unit
$\eta:\id\to L_B$ up to contractible choices.
The category $\bMod_A$ carries an action of $\cC$ given by
$(X,M)\mapsto X\otimes M$. The localization is compatible
with this action, through natural equivalences
\[
L_B(X\otimes M)
\simeq (X\otimes M)\otimes_A B
\simeq X\otimes L_B(M).
\]
This compatibility, including the associativity and unit
compatibilities, is what we mean by a $\cC$-linear
structure on $L_B$.

The same formula shows that tensoring with $X\in\cC$
preserves both local objects and local equivalences.
In particular,
\[
X\otimes\eta_M:X\otimes M\longrightarrow X\otimes L_B(M)
\]
is a local equivalence with local target. The universal
property of reflection therefore determines the
equivalence $L_B(X\otimes M)\simeq X\otimes L_B(M)$
compatibly with the localization unit, up to contractible
choices. It also determines its associativity and unit
compatibilities. Thus the kernel determines the
localization as a $\cC$-linear functor, together with
its monad unit and multiplication.

By the $\cC$-linear Eilenberg--Watts theorem
\cite[Theorem~4.8.4.1]{HA}, this functor corresponds to
the $A$-bimodule $L_B(A)\simeq B$. Its monad unit and
multiplication correspond to
\[
A\longrightarrow B,
\qquad
B\otimes_A B\longrightarrow B,
\]
with their algebra compatibilities. Hence the kernel
determines $B$ as an associative algebra under $A$.

\begin{prop}\label{prop:well_copowered}
Under these hypotheses, for each $A$ there is only a set of equivalence classes of homotopical epimorphisms $A\to B$. 
\end{prop}
\begin{proof}
The assignment $f\mapsto\mathcal K_B$ takes values in a set of smashing subcategories. Two maps with the same kernel have equivalent pointed $\cC$-linear localization monads and hence equivalent algebras under $A$, by the preceding argument. The claim follows from Theorem~\ref{thm:main}.
\end{proof}

\appendix
\section{An alternative proof in the closed case}\label{app:closed_proof}

Assume the hypotheses of Theorem~\ref{thm:main}(2).
Assume additionally that, for every $X\in\cC$, the functor
$-\otimes X$ admits a right adjoint, denoted
$\underline{\Hom}_\cC(X,-)$. Thus
\[
\Map_\cC(Y,\underline{\Hom}_\cC(X,Z))
\simeq\Map_\cC(Y\otimes X,Z)
\]
naturally in $Y,Z$. We give a second proof of the implication
from homotopical to homological epimorphisms.

\begin{lemma}\label{lemma:core_additive}
Let $V:\cE\to\cF$ be an additive functor between additive $\infty$-categories. If its functor on maximal subgroupoids is fully faithful, then $V$ is fully faithful.
\end{lemma}
\begin{proof}
For $M,N\in\cE$, the whole mapping space $\Map_\cE(M,N)$ is a retract of $\operatorname{Aut}_\cE(M\oplus N)$. The maps are
\[
g\longmapsto
\begin{pmatrix}\id_M&0\\g&\id_N\end{pmatrix},
\qquad
h\longmapsto p_Nhi_M.
\]
The inverse of the displayed automorphism replaces $g$ by $-g$, and the composite of the two maps is the identity. Addition, negation, and biproduct maps define these as maps of spaces, not merely operations on homotopy classes. They are natural under additive functors. The restriction map
\[
\Map_\cE(M,N)\longrightarrow\Map_\cF(VM,VN)
\]
is therefore a retract, in the arrow category of spaces, of the corresponding map of automorphism spaces. The latter is an equivalence by full faithfulness on maximal subgroupoids. Retracts of equivalences are equivalences.
\end{proof}

\begin{prop}\label{prop:closed_proof}
Under the additional right adjoint hypothesis above, every homotopical epimorphism $f:A\to B$ is homological.
\end{prop}
\begin{proof}
For $X\in\cC$, its internal endomorphism object
\[
\operatorname{End}_\cC(X)=\underline{\Hom}_\cC(X,X)
\]
is an associative algebra under composition. The space of left $R$-module structures on the fixed object $X$ is
\[
\operatorname{Act}_R(X)
 \simeq\Map_{\bAlg(\cC)}(R,\operatorname{End}_\cC(X));
\]
see \cite[\S4.7.1]{HA}. This classifies the entire coherent action. Since $f$ is homotopical, Lemma~\ref{lemma:homotopical_epi_mapping_spaces}, applied to $\operatorname{End}_\cC(X)$, shows that the restriction
\[
\operatorname{Act}_B(X)\longrightarrow\operatorname{Act}_A(X)
\]
is $(-1)$-truncated for every $X$. 

For an $\infty$-category $\cE$, write $\cE^\simeq$ for its maximal $\infty$-groupoid, consisting of all objects of $\cE$ and the equivalences between them. Restriction of scalars induces a map of spaces
\[
U_f^\simeq:
({}_B\bMod)^\simeq\longrightarrow({}_A\bMod)^\simeq.
\]
Its homotopy fiber over an $A$-module $Q$ is the space
of $B$-module structures extending the given
$A$-module structure, including an identification
of their restriction with $Q$. If $X$ is the underlying
object of $Q$ and $\rho_Q\in\operatorname{Act}_A(X)$
its action, this fiber identifies with
\[
\operatorname{fib}_{\rho_Q}
\bigl(
\operatorname{Act}_B(X)\longrightarrow
\operatorname{Act}_A(X)
\bigr).
\]
By the preceding argument, this space is either empty
or contractible. Thus $U_f^\simeq$ is $(-1)$-truncated,
and hence fully faithful as a functor of
$\infty$-groupoids.

The loop space of $({}_R\bMod)^\simeq$ at an $R$-module
$P$ is its automorphism space
$\operatorname{Aut}_R(P)$, consisting of the
$R$-linear self-equivalences of $P$. Consequently,
full faithfulness of $U_f^\simeq$ gives
\[
\operatorname{Aut}_B(P)
\simeq
\operatorname{Aut}_A(U_fP)
\]
for every $B$-module $P$.
Restriction preserves biproducts, and both module categories are additive. Lemma~\ref{lemma:core_additive} makes $U_f:{}_B\bMod\to{}_A\bMod$ fully faithful. Proposition~\ref{prop:fully_faithful} now implies that $B\otimes_A B\to B$ is an equivalence.
\end{proof}

Together with the first implication of Theorem~\ref{thm:main}, this proves the equivalence in the closed additive case.

\end{document}